\documentclass[12pt]{amsart}

\usepackage{graphicx}

\usepackage{subcaption} 
\usepackage{xfrac}   
\usepackage{faktor}
\usepackage{tikz-cd}
\usepackage{mathrsfs}
\usepackage{hyperref}
\usepackage{amsmath}
\usepackage{amssymb}
\hypersetup{bookmarks=true,
	unicode=true,
	colorlinks=true,
	citecolor=black,
	linkcolor=black,
	urlcolor=black,
	plainpages=false,
	pdfpagelabels=true}

 \usepackage{url}	
 \allowdisplaybreaks 

\usepackage{tikz-cd}
\usepackage{pgf}

\usepackage{xcolor}

\usepackage{comment} 

\usepackage[all]{xy}

\newtheorem{teo}{Theorem}[section]

\newtheorem{thm}[teo]{Theorem}
\newtheorem{cor}[teo]{Corollary}

\newtheorem{lemma}[teo]{Lemma}
\newtheorem{prop}[teo]{Proposition}

\theoremstyle{definition}

\newtheorem{defi}[teo]{Definition}

\newtheorem{ex}[teo]{Example}

\theoremstyle{remark}

\newtheorem{rem}{Remark}

\numberwithin{figure}{section}

\newcommand{\CC}{\mathcal{C}}

\newcommand{\DD}{\mathcal{D}}

\newcommand{\cat}{\mathop{\mathrm{cat}}}

\newcommand{\TC}{\mathrm{TC}}

\newcommand{\Dc}{\mathcal{D}}

\newcommand{\Vv}{\mathbb{V}}

\newcommand{\op}{\mathrm{op}}

\newcommand{\Top}{\mathrm{Top}}

\newcommand{\X}{\mathbb{X}}

\newcommand{\Y}{\mathbb{Y}}

\newcommand{\V}{\mathbb{V}}
\newcommand{\T}{\mathbb{T}}

\newcommand{\Hom}{\mathrm{Hom}}

\newcommand{\Z}{\mathbb{Z}}
\newcommand{\Obj}{\mathrm{Obj}}
\newcommand{\Set}{\mathrm{Set}}

\newcommand{\Aa}{\mathbb{A}}
\newcommand{\Ee}{\mathbb{E}}
\newcommand{\Bb}{\mathbb{B}}
\newcommand{\Uu}{\mathbb{U}}

\begin{document}

\title[Motion planning for diagrams of spaces]{Motion planning and topological complexity for diagrams of spaces}
\thanks{}

\author[Isaac Carcacía-Campos]{%
	Isaac Carcacía-Campos
}

 \address{%
           	Isaac Carcacía-Campos
            \\
              Departamento de Matem\'aticas, Universidade de Santiago de Compostela, 15782-SPAIN}
               \email{isaac.c.campos@usc.es}


\begin{abstract} 
A diagram of spaces describes a system of interacting state spaces. Motion planning in such a system is not merely an objectwise problem since the selected paths must form a natural family. This leads to a notion of topological complexity for diagrams, defined as the sectional category of their endpoint evaluation map.

The role of points in this setting is played by orbits, namely diagrams whose colimit is a point. Different choices of admissible orbit shapes give rise to different levels of categorical compression. We consider the families of representable, discrete and arbitrary topological orbits, and the corresponding orbit-relative Lusternik--Schnirelmann categories. Their relationship with sectional category provides upper bounds for the topological complexity of a diagram, while its values and suitable inverse limits provide lower bounds. For diagrams indexed by a finite discrete group, the construction recovers equivariant LS-category, equivariant sectional category and equivariant topological complexity.
\end{abstract}


\maketitle


\section*{Introduction}

Topology studies topological properties. Beneath this apparent
triviality lies a certain wisdom. Once topological properties are
understood as invariants under an appropriate notion of equivalence,
such as homeomorphism or homotopy equivalence, a more meaningful
picture of topology emerges. Following this perspective, one may study
numerical invariants that encode properties of intrinsic interest and,
at the same time, have significant applications.

A central example is the topological complexity introduced by Farber
\cite{FARBER}. It measures the discontinuity inherent in the motion
planning problem: how many open subsets are needed in order to assign,
continuously on each of them, a path joining any given pair of points?
Topological complexity belongs to a broader family of invariants defined
through open covers and local homotopical conditions. Among them are the
Lusternik--Schnirelmann category, introduced in
\cite{LS-CATEGORY} and developed extensively in \cite{CLOT}, and the
\v{S}varc genus, or sectional category, originally introduced for
fibrations in \cite{Svarc}. The latter measures how far a map is from
admitting a global section and provides a common framework for both
Lusternik--Schnirelmann category and topological complexity.

These invariants have also been extended to spaces equipped with group
actions. Equivariant LS-category was studied by Marzantowicz
\cite{Marzantowwicz} and Colman \cite{Colman}, while Colman and Grant
introduced equivariant sectional category and equivariant topological
complexity in \cite{Grant-Colman-Equivariant}. 

Other extensions incorporate additional data into motion planning.
Parametrized topological complexity treats external conditions as part
of the input
\cite{CohenFarberWeinberger,FarberWeinberger}, while bivariate
topological complexity coordinates trajectories through a common target
\cite{GarciaCalcinesVilches}. These approaches reflect a broader shift
from isolated configuration spaces to systems involving parameters,
symmetries or compatibility constraints.

The purpose of this article is to develop a theory of topological
complexity for diagrams of spaces. A diagram may describe a system of
interacting state spaces in which the continuous maps associated with
the indexing category encode how states and motions at different
components are related. A motion planner for such a system must assign
paths simultaneously and naturally: applying a structure map to a
selected path must produce the path selected for the corresponding pair
of image states. The resulting problem is therefore not determined by
the individual state spaces alone.

A discrete group can be regarded as a category with one object, whose
morphisms are the elements of the group. Accordingly, an action of a
discrete group $G$ on a topological space is equivalently a functor from
$G$ to $\Top$. This suggests replacing $G$ by an arbitrary small
category $\CC$ and studying $\CC$-spaces, that is, functors
$\X\colon\CC\to\Top$. For a non-discrete topological group, an enriched
formulation would be required. We work throughout with
ordinary small indexing categories; an enriched extension lies beyond
the scope of this paper.

The homotopy theory of diagrams of spaces was developed by Dror Farjoun
and Zabrodsky
\cite{Farjou_Zabrodsky,Farjoun_Hom,FarjounBig}. A central role in this
theory is played by orbits, namely $\CC$-spaces whose colimit is a point.
Orbit-based model structures and the size of their classes of orbits
were further studied in \cite{Chorny2003}, and an Elmendorf-type theorem
for diagrams was established in \cite{Housden}. In the present work,
orbits provide the generalized points through which categorical
compressions are allowed to factor.

Rather than fixing a single collection of orbit shapes, we work relative
to orbit families containing all representable $\CC$-spaces. Three
canonical families are considered: representable orbits, discrete
orbits given by indecomposable $\CC$-sets, and arbitrary topological
orbits. They form an increasing sequence
\[
\mathcal O_{\mathrm{rep}}(\CC)
\subseteq
\mathcal O_{\mathrm{disc}}(\CC)
\subseteq
\mathcal O_{\mathrm{top}}(\CC)
\]
and give rise to three orbit-relative
Lusternik--Schnirelmann categories satisfying
\[
\cat_{\CC}^{\mathrm{top}}(\X)
\leq
\cat_{\CC}^{\mathrm{disc}}(\X)
\leq
\cat_{\CC}^{\mathrm{rep}}(\X).
\]
This construction is a diagrammatic analogue of the category relative
to a class of spaces introduced by Clapp and Puppe
\cite{ClappPuppe1986}. The three invariants need not coincide: we give
an example in which both inequalities above are strict. Thus, enlarging
the class of admissible orbit shapes genuinely changes the possible
categorical compressions of a diagram.

The sectional category of a morphism of $\CC$-spaces is defined using
open covers and homotopy local sections that are natural with respect to
$\CC$. The topological complexity of a $\CC$-space $\X$ is then
\[
\TC_{\CC}(\X)=\sec_{\CC}(\pi_{\X}),
\]
where
\[
\pi_{\X}\colon\X^{\mathbb I}\longrightarrow\X\times\X
\]
is the objectwise endpoint evaluation map. Its local sections are
precisely compatible local motion planners. We prove that
$\pi_{\X}$ has the homotopy lifting property, so homotopy local sections
may equivalently be replaced by strict local sections.

The principal upper bounds arise from the interaction between
sectional category and orbit-relative LS-category. For an orbit family
$\mathcal O$, we introduce a corresponding notion of
$\mathcal O$-surjectivity and prove that every
$\mathcal O$-surjective morphism
$p\colon\Ee\to\Bb$ satisfies
\[
\sec_{\CC}(p)
\leq
\cat_{\CC}^{\mathcal O}(\Bb).
\]
For endpoint evaluation, $\mathcal O$-surjectivity is equivalent to the
path-connectivity of the spaces of generalized points
\[
\Hom_{\CC}(\T,\X),
\qquad
\T\in\mathcal O.
\]
Consequently, every $\mathcal O$-path-connected $\CC$-space satisfies
\[
\TC_{\CC}(\X)
\leq
\cat_{\CC}^{\mathcal O}(\X\times\X).
\]
The canonical families thereby yield increasingly strong connectivity
conditions and potentially sharper upper bounds.

The topological complexity of a diagram also admits ordinary lower
bounds. Evaluation at each object gives
\[
\sup_{c\in\Obj(\CC)}
\TC\bigl(\X(c)\bigr)
\leq
\TC_{\CC}(\X).
\]

For finite-group actions, the discrete orbit family consists of the
homogeneous spaces $G/H$, and the corresponding constructions recover
equivariant LS-category, equivariant sectional category and equivariant
topological complexity. 

Under a finiteness condition and the existence of a weakly initial
object, the inverse limit gives a further lower bound. We show that this
hypothesis is essential and that no analogous lower bound holds in
general for the colimit. Finally, we establish change-of-base inequalities for the three types
of invariants.

The paper is organized as follows.
Section~\ref{sec:homotopy-diagrams} recalls the homotopy theory and
topological enrichment of $\CC$-spaces.
Section~\ref{sec:orbits} studies orbit families and canonical orbital
points.
Section~\ref{sec:orbit-relative-category} introduces orbit-relative
LS-category and distinguishes the three canonical variants.
Section~\ref{sec:sectional-category} develops sectional category for
diagrams and its comparison with orbit-relative LS-category.
Section~\ref{sec:topological-complexity} establishes the main results
on diagrammatic topological complexity, including its upper and lower
bounds and the examples concerning limits and colimits.
Finally, Section~\ref{sec:change-of-base} treats change of base and
restriction to subcategories.

Throughout the paper, $\Top$ denotes the category of compactly generated
topological spaces, and mapping spaces are endowed with the compactly
generated compact-open topology. We write $X\simeq Y$ when two spaces
are homotopy equivalent and $X\cong Y$ when they are homeomorphic. All
categories are assumed to be small, and $\Obj(\CC)$ denotes the set of
objects of $\CC$.

\section{Preliminaries on diagrams of spaces}
\label{sec:homotopy-diagrams}
\label{sec:mapping-spaces}

We write $\Top^{\CC}$ for the
category of covariant functors from $\CC$ to $\Top$.

\begin{defi}
A \emph{$\CC$-space} is a functor $\X\colon\CC\to\Top$. A morphism
$f\colon\X\to\Y$ of $\CC$-spaces is a natural transformation and will
also be called a \emph{$\CC$-map}.
\end{defi}

Thus, a $\CC$-map $f\colon\X\to\Y$ consists of continuous maps
$f_c\colon\X(c)\to\Y(c)$ satisfying
$f_d\circ\X(\alpha)=\Y(\alpha)\circ f_c$ for every morphism
$\alpha\colon c\to d$ in $\CC$.

Since $\Top$ is complete and cocomplete, so is $\Top^{\CC}$, and its
limits and colimits are computed objectwise. In particular,
$(\X\times\Y)(c)=\X(c)\times\Y(c)$ and
$(\X\times\Y)(\alpha)=\X(\alpha)\times\Y(\alpha)$.

Every topological space $K$ will also be regarded as a constant
$\CC$-space. This defines the diagonal functor
$\Delta\colon\Top\to\Top^{\CC}$, although we shall usually omit
$\Delta$ from the notation.

\begin{defi}
A $\CC$-space $\X$ is a \emph{$\CC$-subspace} of a $\CC$-space $\Y$
if $\X(c)$ is a subspace of $\Y(c)$ for every $c\in\Obj(\CC)$ and
$\X(\alpha)$ is the restriction of $\Y(\alpha)$ for every morphism
$\alpha$ of $\CC$. The objectwise inclusions then define a $\CC$-map
$\iota\colon\X\hookrightarrow\Y$.
\end{defi}

\subsection{Limits, colimits and homotopies}

The limit and colimit of a $\CC$-space are topological spaces obtained
by taking the corresponding construction over the indexing category.
The diagonal functor occurs in the adjunctions
\[
\varinjlim_{\CC}\dashv\Delta\dashv\varprojlim_{\CC}.
\]
Consequently, for every $\CC$-space $\X$ and every topological space
$K$, there are natural bijections
\[
\Hom_{\Top}\left(\varinjlim_{\CC}\X,K\right)
\cong
\Hom_{\Top^{\CC}}(\X,K),
\]
and
\[
\Hom_{\Top^{\CC}}(K,\X)
\cong
\Hom_{\Top}\left(K,\varprojlim_{\CC}\X\right).
\]

In particular, every $\CC$-map $K\to\X$ from a constant diagram factors
uniquely through the universal cone
$\varprojlim_{\CC}\X\to\X$, while every $\CC$-map $\X\to K$ factors
uniquely through the universal cocone
$\X\to\varinjlim_{\CC}\X$.

Let $\mathbb I=[0,1]$, regarded as a constant $\CC$-space.

\begin{defi}\label{def:C-homotopy}
Let $f,g\colon\X\to\Y$ be $\CC$-maps. A \emph{$\CC$-homotopy} from
$f$ to $g$ is a $\CC$-map
$H\colon\X\times\mathbb I\to\Y$ such that $H_0=f$ and $H_1=g$, where
$H_t=H\circ i_t$ and $i_t(c)(x)=(x,t)$. If such a homotopy exists, we
write $f\simeq_{\CC}g$.
\end{defi}

Equivalently, a $\CC$-homotopy consists of homotopies
$H^c\colon\X(c)\times\mathbb I\to\Y(c)$ satisfying
\[
\Y(\alpha)\circ H^c
=
H^d\circ\bigl(\X(\alpha)\times1_{\mathbb I}\bigr)
\]
for every $\alpha\colon c\to d$. 

Thus, all homotopies considered here are strictly natural.


\begin{defi}
Two $\CC$-spaces $\X$ and $\Y$ are \emph{$\CC$-homotopy equivalent} if
there are $\CC$-maps $f\colon\X\to\Y$ and $g\colon\Y\to\X$ such that
$g\circ f\simeq_{\CC}1_{\X}$ and
$f\circ g\simeq_{\CC}1_{\Y}$. In that case, we write
$\X\simeq_{\CC}\Y$.
\end{defi}

\subsection{Mapping spaces}

The category $\Top^{\CC}$ is naturally enriched over $\Top$.

\begin{defi}\label{def:mapping-space-diagrams}
Let $\X$ and $\Y$ be $\CC$-spaces. The \emph{mapping space}
$\Hom_{\CC}(\X,\Y)$ is the subspace of
\[
\prod_{c\in\Obj(\CC)}
\Hom_{\Top}\bigl(\X(c),\Y(c)\bigr)
\]
formed by the families $(f_c)_{c\in\Obj(\CC)}$ satisfying
$\Y(\alpha)\circ f_c=f_d\circ\X(\alpha)$ for every
$\alpha\colon c\to d$.

Products, subspaces and mapping spaces are taken in the category of
compactly generated spaces.
\end{defi}

Composition is continuous and defines a bifunctor
\[
\Hom_{\CC}(-,-)\colon
(\Top^{\CC})^{\op}\times\Top^{\CC}\longrightarrow\Top.
\]

For a topological space $K$ and a $\CC$-space $\X$, let $\X^K$ denote
the $\CC$-space given by
\[
\X^K(c)=\Hom_{\Top}(K,\X(c)),
\qquad
\X^K(\alpha)(u)=\X(\alpha)\circ u.
\]
We also write $\X\times K$ for the product of $\X$ with the constant
diagram $K$.

\begin{prop}\label{prop:adjoint}
For $\CC$-spaces $\X,\Y$ and a topological space $K$, there are natural
homeomorphisms
\[
\Hom_{\CC}(\X,\Y^K)
\cong
\Hom_{\CC}(\X\times K,\Y)
\cong
\Hom_{\Top}\bigl(K,\Hom_{\CC}(\X,\Y)\bigr).
\]
\end{prop}

\begin{proof}
The ordinary exponential homeomorphisms
\[
\Hom_{\Top}\bigl(\X(c),\Hom_{\Top}(K,\Y(c))\bigr)
\cong
\Hom_{\Top}\bigl(\X(c)\times K,\Y(c)\bigr)
\]
are natural in $c$ and restrict to the subspaces of natural
transformations. The second homeomorphism is the corresponding
tensor--hom adjunction for the enriched functor category
$\Top^{\CC}$.
\end{proof}

The preceding adjunction identifies the inverse limit of a $\CC$-space
with a mapping space:
\[
\Hom_{\CC}(*,\X)\cong\varprojlim_{\CC}\X.
\]
Explicitly, both spaces consist of compatible families
$(x_c)_{c\in\Obj(\CC)}$ satisfying
$\X(\alpha)(x_c)=x_d$ for every morphism $\alpha\colon c\to d$. If
$\CC=G$ is a group regarded as a one-object category, this is the
fixed-point space of the corresponding $G$-action.

Mapping spaces also preserve products in the second variable. More
precisely, there is a natural homeomorphism
\begin{equation}\label{eq:mapping-space-product}
\Hom_{\CC}(\X,\Y\times\Z)
\cong
\Hom_{\CC}(\X,\Y)\times\Hom_{\CC}(\X,\Z),
\end{equation}
induced by the two projections.

\subsection{Compatibility with homotopies}

\begin{prop}
\label{prop:Hom_preserves_homotopies}
\label{cor:limit-preserves-homotopies}
\label{pro:colimit_preserves_homotopies}
Let $f,g\colon\X\to\Y$ be $\CC$-homotopic.
\begin{enumerate}
    \item For every $\CC$-space $\Aa$, the induced maps
    $\Hom_{\CC}(\Aa,f)$ and $\Hom_{\CC}(\Aa,g)$ are homotopic.
    Consequently, $\X\simeq_{\CC}\Y$ implies
    $\Hom_{\CC}(\Aa,\X)\simeq\Hom_{\CC}(\Aa,\Y)$.

    \item The maps $\varprojlim_{\CC}f$ and
    $\varprojlim_{\CC}g$ are homotopic.

    \item The maps $\varinjlim_{\CC}f$ and
    $\varinjlim_{\CC}g$ are homotopic.
\end{enumerate}
In particular, both limits and colimits preserve
$\CC$-homotopy equivalences.
\end{prop}

\begin{proof}
Let $H\colon\X\times\mathbb I\to\Y$ be a $\CC$-homotopy from $f$ to
$g$. For the first assertion, define
$\overline H(m,t)=H_t\circ m$. Continuity follows from enriched
composition, and the endpoint maps are
$\Hom_{\CC}(\Aa,f)$ and $\Hom_{\CC}(\Aa,g)$.

The second assertion follows by taking $\Aa=*$ and using
$\Hom_{\CC}(*,\X)\cong\varprojlim_{\CC}\X$.

Finally, since $-\times\mathbb I$ is a left adjoint in $\Top$, it
preserves colimits. Hence
\[
\varinjlim_{\CC}(\X\times\mathbb I)
\cong
\left(\varinjlim_{\CC}\X\right)\times\mathbb I,
\]
and applying $\varinjlim_{\CC}$ to $H$ gives a homotopy from
$\varinjlim_{\CC}f$ to $\varinjlim_{\CC}g$.
\end{proof}
\section{Orbit families}
\label{sec:orbits}

A point of a topological space $X$ can be identified with a map
$*\to X$. This observation underlies several elementary notions: a map is
surjective when every point of its codomain lifts, and a space is
path-connected when any two of its points can be joined by a path.

For a $\CC$-space, the constant one-point diagram is generally too
restrictive to play the role of all possible points. Following the orbit
approach to the homotopy theory of diagrams
\cite{FarjounBig,Farjoun_Hom,Farjou_Zabrodsky,Housden}, we regard a
$\CC$-map $O\colon\T\to\X$ from a suitable $\CC$-space $\T$ as a
generalized point of $\X$ of shape $\T$.

\begin{defi}
A \emph{topological $\CC$-orbit} is a $\CC$-space
$\T\colon\CC\to\Top$ such that
$\varinjlim_{\CC}\T\cong *$. We denote by
$\mathcal O_{\mathrm{top}}(\CC)$ the class of all topological
$\CC$-orbits.
\end{defi}

Thus, although an orbit may contain many points at its different values,
all these points determine the same element after passing to the colimit.
This generalizes the transitivity condition for group actions.

Different collections of orbits may detect different features of a
diagram, as already occurs in orbit-based homotopy theories
\cite{Farjoun_Hom,Housden}. We shall therefore work relative to a chosen
family.

\begin{defi}
An \emph{orbit family} for $\CC$ is a subclass
$\mathcal O\subseteq\mathcal O_{\mathrm{top}}(\CC)$ that is closed
under isomorphisms and contains every representable $\CC$-space
$\CC(c,-)$, regarded as an objectwise discrete $\CC$-space.
\end{defi}

The representables ensure that an orbit family detects objectwise
information, since the enriched Yoneda lemma gives
$\Hom_{\CC}(\CC(c,-),\X)\cong\X(c)$.

\subsection{Three canonical families}

A \emph{$\CC$-set} is a functor $S\colon\CC\to\Set$. For more information about $\CC$-set see \cite{Webb2023,CalderonEtAl2026}. We regard every
$\CC$-set as a $\CC$-space by equipping each $S(c)$ with the discrete
topology.

\begin{defi}\label{def:canonical-orbit-families}
We consider the following three orbit families:
\begin{enumerate}
    \item the \emph{representable family}
    \[
    \mathcal O_{\mathrm{rep}}(\CC)
    =
    \{\CC(c,-)\mid c\in\Obj(\CC)\},
    \]
    closed under isomorphisms;

    \item the \emph{discrete family}
    \[
    \mathcal O_{\mathrm{disc}}(\CC)
    =
    \left\{
    S\colon\CC\to\Set
    \ \middle|\
    \varinjlim_{\CC}^{\Set}S\cong *
    \right\},
    \]
    where its elements are regarded as discrete $\CC$-spaces;

    \item the \emph{topological family}
    \[
    \mathcal O_{\mathrm{top}}(\CC)
    =
    \left\{
    \T\colon\CC\to\Top
    \ \middle|\
    \varinjlim_{\CC}^{\Top}\T\cong *
    \right\}.
    \]
\end{enumerate}
\end{defi}

Every representable is an orbit. Indeed, for each
$f\colon c\to d$, the elements $1_c\in\CC(c,c)$ and
$f\in\CC(c,d)$ have the same image in the colimit because
$\CC(c,f)(1_c)=f$. Hence
$\varinjlim_{\CC}\CC(c,-)\cong *$.

The colimit of a $\CC$-set $S$ is the quotient of
$\coprod_{c\in\Obj(\CC)}S(c)$ by the equivalence relation generated by
$s\sim S(\alpha)(s)$. Therefore, $S$ is a discrete orbit precisely when
it is nonempty and all its elements belong to a single equivalence class.

\begin{prop}\label{prop:discrete-orbit-indecomposable}
Let $S\colon\CC\to\Set$ be a nonempty $\CC$-set. The following
conditions are equivalent:
\begin{enumerate}
    \item $S$ is a discrete $\CC$-orbit;
    \item $\varinjlim_{\CC}^{\Set}S\cong *$;
    \item $S$ is indecomposable in $\Set^{\CC}$;
    \item the Grothendieck construction $\int_{\CC}S$ is connected.
\end{enumerate}
\end{prop}

\begin{proof}
This follows from the standard decomposition of a $\CC$-set into its
indecomposable components; see \cite{Webb2023}. These
components correspond to the connected components of
$\int_{\CC}S$, and hence to the elements of
$\varinjlim_{\CC}^{\Set}S$.
\end{proof}

Thus, discrete $\CC$-orbits are the categorical analogues of transitive
$G$-sets.

Since every representable orbit is discrete and every
discrete orbit is topological, we have
\[
\mathcal O_{\mathrm{rep}}(\CC)
\subseteq
\mathcal O_{\mathrm{disc}}(\CC)
\subseteq
\mathcal O_{\mathrm{top}}(\CC).
\]

The three families provide increasingly general notions of points.
Representable points are ordinary points of individual spaces $\X(c)$;
discrete orbital points are compatible indecomposable systems of points;
and topological orbital points may themselves carry nontrivial topology.

When $\CC=G$ is a discrete group, the discrete orbits are precisely the
transitive $G$-sets $G/H$, while the unique representable orbit is the
regular $G$-set $G$. We return to the equivariant interpretation in
Section~\ref{subsec:equivariant-case}.

\subsection{Examples}

The full class of topological orbits can be large, even for a very simple
indexing category.

\begin{ex}\label{ex:interval}
Let $\mathcal I$ be the category $0\to1$. Since $1$ is terminal,
$\varinjlim_{\mathcal I}\X\cong\X(1)$ for every $\mathcal I$-space
$\X$. Hence the topological $\mathcal I$-orbits are precisely the
diagrams $X\to *$, with $X$ arbitrary. Thus,
$\mathcal O_{\mathrm{top}}(\mathcal I)$ contains a copy of the class of
all topological spaces and need not be a set. This is an instance of the
size phenomenon studied in \cite{Chorny2003}. If $X$ is not discrete we have a non-discrete orbit, so the
inclusion
$\mathcal O_{\mathrm{disc}}(\mathcal I)
\subseteq\mathcal O_{\mathrm{top}}(\mathcal I)$
is strict.
\end{ex}

\begin{rem}
The constant one-point $\CC$-space is an orbit if and only if $\CC$ is
nonempty and connected, since
$\varinjlim_{\CC}*\cong\pi_0(\CC)$ as a discrete space.
\end{rem}

The values of a topological orbit need not be contractible.

\begin{ex}\label{ex:circunference_Category_orbit}
Let $\mathcal S$ be generated by two parallel morphisms
\[
\begin{tikzcd}
a
  \arrow[r,"f_1",bend left]
  \arrow[r,"f_2"',bend right]
&
b,
\end{tikzcd}
\]
and consider the $\mathcal S$-space
\[
\begin{tikzcd}
\mathbb S^1
  \arrow[r,"1_{\mathbb S^1}",bend left]
  \arrow[r,"{(1,0)}"',bend right]
&
\mathbb S^1,
\end{tikzcd}
\]
where $(1,0)$ denotes the constant map with value $(1,0)$. Its colimit
is the coequalizer of the identity and a constant map, and is therefore
a point. The diagram is consequently a topological orbit, although both
of its values are noncontractible. 
\end{ex}

\subsection{Orbital points and canonical orbits}

\begin{defi}
Let $\mathcal O$ be an orbit family and let $\X$ be a $\CC$-space. An
\emph{$\mathcal O$-point of $\X$} is a $\CC$-map
$O\colon\T\to\X$ with $\T\in\mathcal O$. For
$\mathcal O=\mathcal O_{\mathrm{top}}(\CC)$, we simply call $O$ an
\emph{$\X$-orbit}.
\end{defi}

For a fixed shape $\T$, the mapping space $\Hom_{\CC}(\T,\X)$ is the
space of generalized points of $\X$ of shape $\T$. This interpretation
will later be used to define orbit-relative surjectivity and
path-connectivity.

Although $\mathcal O_{\mathrm{top}}(\CC)$ may be a proper class, the
orbits mapping to a fixed $\CC$-space factor through a collection indexed
by the points of its colimit.

\begin{defi}\label{def:canonical-topological-orbit}
Let $\X$ be a $\CC$-space and let
$x\in\varinjlim_{\CC}\X$. We define $T_x^{\mathrm{top}}$ as the pullback
\[
\begin{tikzcd}
\T_x^{\mathrm{top}}
  \arrow[r,"O_x^{\mathrm{top}}"]
  \arrow[d]
&
\X
  \arrow[d,"\pi_{\X}"]
\\
*
  \arrow[r,"x"']
&
\varinjlim_{\CC}\X.
\end{tikzcd}
\]
Thus, $\T_x^{\mathrm{top}}(c)$ is the subspace
$\{z\in\X(c)\mid\pi_{\X,c}(z)=x\}$ of $\X(c)$.
\end{defi}

\begin{lemma}\label{lem:canonical-topological-orbit}
For every $x\in\varinjlim_{\CC}\X$, the $\CC$-space
$\T_x^{\mathrm{top}}$ is a topological orbit.
\end{lemma}

\begin{proof}
The values of $\T_x^{\mathrm{top}}$ consist precisely of the
representatives of the class $x$. This class has at least one
representative, and any two such representatives are related by the
equivalence relation generated by the structure maps of $\X$.
Consequently, the colimit of $\T_x^{\mathrm{top}}$ has exactly one
point.
\end{proof}

Hence we can call $\T_x^{\mathrm{top}}$ the \emph{canonical topological orbit of $\X$ over $x$}.

\begin{prop}\label{pro:Orbits_point}
Let $O\colon\T\to\X$ be a topological $\X$-orbit. There are a unique
point $x\in\varinjlim_{\CC}\X$, determined by
$\varinjlim_{\CC}O$, and a unique $\CC$-map
$f\colon\T\to\T_x^{\mathrm{top}}$ such that
$O_x^{\mathrm{top}}\circ f=O$.
\end{prop}

\begin{proof}
Since $\varinjlim_{\CC}\T\cong *$, applying the colimit functor to $O$
determines a point $x\colon *\to\varinjlim_{\CC}\X$. The square
\[
\begin{tikzcd}
\T
  \arrow[r,"O"]
  \arrow[d]
&
\X
  \arrow[d,"\pi_{\X}"]
\\
*
  \arrow[r,"x"']
&
\varinjlim_{\CC}\X
\end{tikzcd}
\]
commutes, and the result follows from the universal property of the
pullback.
\end{proof}

Consequently, a lifting property against all topological orbital points
of $\X$ can be tested on the maps
$O_x^{\mathrm{top}}\colon\T_x^{\mathrm{top}}\to\X$, with
$x\in\varinjlim_{\CC}\X$.

There is an analogous construction for discrete orbital points.

\begin{defi}\label{def:canonical-discrete-orbit}
Let $U\X\colon\CC\to\Set$ be the underlying $\CC$-set of $\X$ and let
$x\in\varinjlim_{\CC}^{\Set}U\X$. The \emph{canonical discrete orbit of
$\X$ over $x$} is
\[
\T_x^{\mathrm{disc}}
=
U\X\times_{\varinjlim_{\CC}^{\Set}U\X}*.
\]
We regard it as a discrete $\CC$-space and denote its canonical map to
$\X$ by $O_x^{\mathrm{disc}}$. Objectwise,
$\T_x^{\mathrm{disc}}(c)=\{z\in\X(c)\mid[z]=x\}$, equipped with the
discrete topology.
\end{defi}

\begin{prop}\label{prop:discrete-orbits-point}
Every discrete orbital point $O\colon\T\to\X$ factors uniquely through
$O_x^{\mathrm{disc}}\colon\T_x^{\mathrm{disc}}\to\X$ for the point
$x\in\varinjlim_{\CC}^{\Set}U\X$ determined by
$\varinjlim_{\CC}^{\Set}O$.
\end{prop}

\begin{proof}
Since $\varinjlim_{\CC}^{\Set}\T\cong *$, the underlying natural
transformation of $O$ determines a point
$x\colon *\to\varinjlim_{\CC}^{\Set}U\X$. The universal property of the
defining pullback gives a unique natural transformation
$\T\to\T_x^{\mathrm{disc}}$. It is continuous because both diagrams are
objectwise discrete.
\end{proof}

\section{Orbit-relative LS-category}
\label{sec:orbit-relative-category}

The classical Lusternik--Schnirelmann category measures the minimum
number of open subsets required to cover a space so that each inclusion
is null-homotopic. For a $\CC$-space, we replace the one-point space by
the generalized points introduced in Section~\ref{sec:orbits}. After
fixing an orbit family $\mathcal O$, categorical open subspaces will be
those that can be compressed, inside the ambient diagram, through an
orbit in $\mathcal O$. This is a diagrammatic instance of the category
relative to a class of spaces introduced by Clapp and Puppe
\cite{ClappPuppe1986}.

\begin{defi}\label{def:open-C-subspaces}
Let $\X$ be a $\CC$-space.
\begin{enumerate}
    \item An \emph{open $\CC$-subspace} of $\X$ is a $\CC$-subspace
    $\Uu\subseteq\X$ such that $\Uu(c)$ is open in $\X(c)$ for every
    $c\in\Obj(\CC)$.

    \item A collection $\{\Uu_i\}_{i\in I}$ of open $\CC$-subspaces is
    an \emph{open cover} of $\X$ if
    $\X(c)=\bigcup_{i\in I}\Uu_i(c)$ for every $c\in\Obj(\CC)$.

    \item If $g\colon\X\to\Y$ is a $\CC$-map and
    $\Uu\subseteq\Y$ is open, its inverse image is the open
    $\CC$-subspace defined by
    $g^{-1}(\Uu)(c)=g_c^{-1}(\Uu(c))$.
\end{enumerate}
\end{defi}

The structure maps of an open $\CC$-subspace are the restrictions of
those of the ambient diagram. Similarly, the naturality of $g$ and the
stability of $\Uu$ under the structure maps of $\Y$ ensure that
$g^{-1}(\Uu)$ is a $\CC$-subspace of $\X$.

\begin{lemma}\label{lem:inverse_image_covers}
If $g\colon\X\to\Y$ is a $\CC$-map and
$\{\Uu_i\}_{i\in I}$ is an open cover of $\Y$, then
$\{g^{-1}(\Uu_i)\}_{i\in I}$ is an open cover of $\X$.
\end{lemma}

\begin{proof}
For every $c\in\Obj(\CC)$, the sets
$g_c^{-1}(\Uu_i(c))$ are open in $\X(c)$ and
\[
\bigcup_{i\in I}g_c^{-1}(\Uu_i(c))
=
g_c^{-1}\left(\bigcup_{i\in I}\Uu_i(c)\right)
=
\X(c).
\]
\end{proof}

\subsection{Categorical open subspaces}

\begin{defi}\label{def:O-categorical-open}
Let $\mathcal O$ be an orbit family and let $\Uu$ be an open
$\CC$-subspace of $\X$. We say that $\Uu$ is
\emph{$\mathcal O$-categorical in $\X$} if there exist
$\T\in\mathcal O$ and $\CC$-maps
$a\colon\Uu\to\T$ and $O\colon\T\to\X$ such that
\[
\iota_{\Uu}\simeq_{\CC}O\circ a,
\]
where $\iota_{\Uu}\colon\Uu\hookrightarrow\X$ is the inclusion.
\end{defi}

Thus, an $\mathcal O$-categorical open subspace can be deformed within
$\X$ to a generalized point whose shape belongs to $\mathcal O$.

\begin{rem}
The factorization through $\T$ is part of the definition. It is not
enough to require that the final map of a homotopy have image contained
in $O(\T)$. This distinction is relevant, for example, when $\T$ is
discrete, since the topology of $\T(c)$ need not coincide with the
subspace topology on the image of $O_c\colon\T(c)\to\X(c)$.
\end{rem}

\begin{defi}\label{def:orbit-relative-category}
Let $\mathcal O$ be an orbit family. The
\emph{$\mathcal O$-Lusternik--Schnirelmann category} of $\X$, denoted
by $\cat_{\CC}^{\mathcal O}(\X)$, is the least integer $n\geq0$ for
which $\X$ admits an open cover
$\{\Uu_0,\ldots,\Uu_n\}$ whose members are
$\mathcal O$-categorical in $\X$. If no such finite cover exists, we
set $\cat_{\CC}^{\mathcal O}(\X)=\infty$.
\end{defi}

We use the reduced convention. Thus,
$\cat_{\CC}^{\mathcal O}(\X)=0$ if and only if
$1_{\X}$ factors, up to $\CC$-homotopy, through a single orbit in
$\mathcal O$.

For the canonical families of Section~\ref{sec:orbits}, we write
\[
\cat_{\CC}^{\mathrm{rep}}(\X),\qquad
\cat_{\CC}^{\mathrm{disc}}(\X),\qquad
\cat_{\CC}^{\mathrm{top}}(\X).
\]
These correspond respectively to compressions through representable,
discrete and arbitrary topological orbits.

If $\CC$ is the terminal category, all three families reduce to the
one-point space and hence
\[
\cat_{\CC}^{\mathrm{rep}}(\X)
=
\cat_{\CC}^{\mathrm{disc}}(\X)
=
\cat_{\CC}^{\mathrm{top}}(\X)
=
\cat(\X).
\]

\subsection{Basic properties}

Enlarging the orbit family provides more possible targets for
categorical compressions.

\begin{prop}\label{prop:monotonicity-orbit-family}
If $\mathcal O\subseteq\mathcal O'$ are orbit families, then
\[
\cat_{\CC}^{\mathcal O'}(\X)
\leq
\cat_{\CC}^{\mathcal O}(\X)
\]
for every $\CC$-space $\X$. In particular,
\[
\cat_{\CC}^{\mathrm{top}}(\X)
\leq
\cat_{\CC}^{\mathrm{disc}}(\X)
\leq
\cat_{\CC}^{\mathrm{rep}}(\X).
\]
\end{prop}

\begin{proof}
Every $\mathcal O$-categorical open subspace is
$\mathcal O'$-categorical. Hence every $\mathcal O$-categorical cover
is also an $\mathcal O'$-categorical cover.
\end{proof}


\begin{prop}\label{prop:category-domination}
Let $f\colon\X\to\Y$ and $g\colon\Y\to\X$ be $\CC$-maps such that
$f\circ g\simeq_{\CC}1_{\Y}$. Then
\[
\cat_{\CC}^{\mathcal O}(\Y)
\leq
\cat_{\CC}^{\mathcal O}(\X).
\]
Consequently, $\CC$-homotopy equivalent diagrams have the same
$\mathcal O$-category.
\end{prop}

\begin{proof}
Let $\{\Uu_0,\ldots,\Uu_n\}$ be an $\mathcal O$-categorical open cover
of $\X$. By Lemma~\ref{lem:inverse_image_covers}, the subspaces
$\V_i=g^{-1}(\Uu_i)$ form an open cover of $\Y$.

For each $i$, choose $\T_i\in\mathcal O$ and maps
$a_i\colon\Uu_i\to\T_i$ and $O_i\colon\T_i\to\X$ with
$\iota_i\simeq_{\CC}O_i\circ a_i$. The restriction of $g$ gives a map
$g_i\colon\V_i\to\Uu_i$ satisfying
$\iota_i\circ g_i=g\circ\iota_i'$, where
$\iota_i'\colon\V_i\hookrightarrow\Y$ is the inclusion. Therefore,
\[
\iota_i'
\simeq_{\CC}
f\circ g\circ\iota_i'
=
f\circ\iota_i\circ g_i
\simeq_{\CC}
f\circ O_i\circ a_i\circ g_i.
\]
Hence $\V_i$ is $\mathcal O$-categorical in $\Y$, proving the
inequality. Applying it to homotopy inverses gives homotopy invariance.
\end{proof}

\subsection{Examples}

\begin{ex}\label{ex:category-two-parallel-arrows}
Let $\mathcal S$ be the category of
Example~\ref{ex:circunference_Category_orbit}, generated by two parallel
arrows, and consider the $\mathcal S$-space
\[
\begin{tikzcd}
\{*\}
  \arrow[r,"0",bend left]
  \arrow[r,"1"',bend right]
&
\mathbb I.
\end{tikzcd}
\]
Then
\[
\cat_{\mathcal S}^{\mathrm{rep}}(\X)
=
\cat_{\mathcal S}^{\mathrm{disc}}(\X)
=
\cat_{\mathcal S}^{\mathrm{top}}(\X)
=
1.
\]

Indeed, the colimit of $\X$ is obtained by identifying the endpoints of
$\mathbb I$. The canonical orbit over an interior point has the form
$\emptyset\rightrightarrows\{t\}$, while the orbit over the common class
of the endpoints is
\[
\begin{tikzcd}
\{*\}
  \arrow[r,"0",bend left]
  \arrow[r,"1"',bend right]
&
\{0,1\}.
\end{tikzcd}
\]
These are respectively the representable orbits
$\mathcal S(b,-)$ and $\mathcal S(a,-)$.

We first show that
$\cat_{\mathcal S}^{\mathrm{top}}(\X)>0$. Otherwise, the identity of
$\X$ would factor up to homotopy through a topological orbit. By
Proposition~\ref{pro:Orbits_point}, it would then factor through one of
the two canonical orbits above.

Factoring through an interior orbit is impossible because its value at
$a$ is empty. Factoring through the endpoint orbit would require a
continuous map $\mathbb I\to\{0,1\}$ satisfying
\[
0\longmapsto0,
\qquad
1\longmapsto1,
\]
by naturality with respect to the two arrows of $\mathcal S$. Such a map
cannot exist because $\mathbb I$ is connected. Hence
$\cat_{\mathcal S}^{\mathrm{top}}(\X)>0$.

For the opposite bound, define open $\mathcal S$-subspaces
$\Uu_0,\Uu_1\subseteq\X$ by
\[
\Uu_0(a)=\Uu_1(a)=\{*\},
\]
\[
\Uu_0(b)=[0,0.3)\cup(0.5,1],
\qquad
\Uu_1(b)=[0,0.6)\cup(0.7,1].
\]
They cover $\X$. Each component of $\Uu_i(b)$ contracts to the endpoint
it contains, and these contractions are compatible with the two
structure maps. Thus, each inclusion $\Uu_i\hookrightarrow\X$ is
homotopic to a map factoring through the representable orbit
$\mathcal S(a,-)$. Therefore,
$\cat_{\mathcal S}^{\mathrm{rep}}(\X)\leq1$.

Finally, Proposition~\ref{prop:monotonicity-orbit-family} gives
\[
1
\leq
\cat_{\mathcal S}^{\mathrm{top}}(\X)
\leq
\cat_{\mathcal S}^{\mathrm{disc}}(\X)
\leq
\cat_{\mathcal S}^{\mathrm{rep}}(\X)
\leq1,
\]
so all three invariants equal $1$.
\end{ex}

\begin{ex}\label{ex:three-orbit-categories-distinct}
Let $\mathcal I=(0\to1)$ and consider the $\mathcal I$-space
\[
\X=
\left(
\mathbb S^1\amalg\mathbb S^1
\longrightarrow
*
\right).
\]
Then
\[
\cat_{\mathcal I}^{\mathrm{top}}(\X)
<
\cat_{\mathcal I}^{\mathrm{disc}}(\X)
<
\cat_{\mathcal I}^{\mathrm{rep}}(\X).
\]
More precisely,
\[
\cat_{\mathcal I}^{\mathrm{top}}(\X)=0,
\qquad
\cat_{\mathcal I}^{\mathrm{disc}}(\X)=1,
\qquad
\cat_{\mathcal I}^{\mathrm{rep}}(\X)=3.
\]

Indeed, $\X$ is a topological orbit, so its topological orbit-relative
category is zero. It does not have discrete category zero, since a map
from either circle to a discrete space is constant and neither identity
$1_{\mathbb S^1}$ is null-homotopic. On the other hand, choosing a
categorical cover by two arcs on each circle and grouping one arc from
each component gives a discrete categorical cover with two members.
Hence
\[
\cat_{\mathcal I}^{\mathrm{disc}}(\X)=1.
\]

Finally, a representably categorical open subspace must factor through
the representable orbit $(*\to*)$. Its inclusion is therefore homotopic
to a single constant map and cannot meet both connected components of
$\X(0)$. Since each circle requires at least two categorical open
subsets, every representable categorical cover has at least four
members. Four such members are obtained by taking two categorical arcs
in each circle. Thus,
\[
\cat_{\mathcal I}^{\mathrm{rep}}(\X)=3.
\]
\end{ex}

\section{Sectional category}
\label{sec:sectional-category}

Sectional category provides a common framework for
Lusternik--Schnirelmann category and topological complexity. Its
definition for maps of $\CC$-spaces is formally the classical one, with
open subsets, maps and homotopies required to be natural.

\begin{defi}\label{def:sectional-category-diagram}
Let $p\colon\Ee\to\Bb$ be a morphism of $\CC$-spaces. An open
$\CC$-subspace $\Uu\subseteq\Bb$ is \emph{sectional for $p$} if there
is a $\CC$-map $s\colon\Uu\to\Ee$ such that
$p\circ s\simeq_{\CC}\iota_{\Uu}$.

The \emph{sectional category} of $p$, denoted by $\sec_{\CC}(p)$, is the
least integer $n\geq0$ such that $\Bb$ admits an open cover
$\{\Uu_0,\ldots,\Uu_n\}$ by sectional subspaces. If no such finite cover
exists, we set $\sec_{\CC}(p)=\infty$.
\end{defi}

We use the reduced convention, so $\sec_{\CC}(p)=0$ precisely when $p$
admits a global homotopy section. When $\CC$ is the terminal category,
this definition reduces to the ordinary sectional category
$\sec(p)$ of a map of topological spaces.

\subsection{Orbit-relative surjectivity}

For maps of spaces, surjectivity means that every point of the codomain
lifts. Since points of a diagram are represented by maps from orbits, the
corresponding notion depends on the chosen orbit family.

\begin{defi}\label{def:O-surjective}
Let $\mathcal O$ be an orbit family. A morphism
$p\colon\Ee\to\Bb$ is \emph{$\mathcal O$-surjective} if
\[
\Hom_{\CC}(\T,p)\colon
\Hom_{\CC}(\T,\Ee)\longrightarrow\Hom_{\CC}(\T,\Bb)
\]
is surjective for every $\T\in\mathcal O$. Equivalently, every
$\mathcal O$-point $O\colon\T\to\Bb$ admits a lift
$\widetilde O\colon\T\to\Ee$ with $p\circ\widetilde O=O$.
\end{defi}

If $\mathcal O\subseteq\mathcal O'$, then every
$\mathcal O'$-surjective map is $\mathcal O$-surjective. In particular,
\[
\mathcal O_{\mathrm{top}}\text{-surjective}
\Longrightarrow
\mathcal O_{\mathrm{disc}}\text{-surjective}
\Longrightarrow
\mathcal O_{\mathrm{rep}}\text{-surjective}.
\]

\begin{prop}
\label{prop:surjectivity-monotone-family}
\label{prop:three-surjectivities}
Let $p\colon\Ee\to\Bb$ be a morphism of $\CC$-spaces.
\begin{enumerate}
    \item The map $p$ is
    $\mathcal O_{\mathrm{rep}}(\CC)$-surjective if and only if every
    component $p_c\colon\Ee(c)\to\Bb(c)$ is surjective.

    \item The map $p$ is
    $\mathcal O_{\mathrm{disc}}(\CC)$-surjective if and only if every
    canonical discrete orbital point
    $O_x^{\mathrm{disc}}\colon\T_x^{\mathrm{disc}}\to\Bb$, with
    $x\in\varinjlim_{\CC}^{\Set}U\Bb$, lifts through $p$.

    \item The map $p$ is
    $\mathcal O_{\mathrm{top}}(\CC)$-surjective if and only if every
    canonical topological orbital point
    $O_x^{\mathrm{top}}\colon\T_x^{\mathrm{top}}\to\Bb$, with
    $x\in\varinjlim_{\CC}\Bb$, lifts through $p$.
\end{enumerate}
\end{prop}

\begin{proof}
The first assertion follows from the enriched Yoneda homeomorphisms
$\Hom_{\CC}(\CC(c,-),\Ee)\cong\Ee(c)$ and
$\Hom_{\CC}(\CC(c,-),\Bb)\cong\Bb(c)$.

For the discrete case, necessity is immediate. Conversely, every
discrete orbital point $O\colon\T\to\Bb$ factors, by
Proposition~\ref{prop:discrete-orbits-point}, as
$O=O_x^{\mathrm{disc}}\circ a$ for some
$a\colon\T\to\T_x^{\mathrm{disc}}$. A lift of
$O_x^{\mathrm{disc}}$ therefore yields a lift of $O$ after composition
with $a$. The topological case follows in the same way from
Proposition~\ref{pro:Orbits_point}.
\end{proof}

Orbit-relative surjectivity is stronger than surjectivity on colimits:
the latter only lifts the class of a point, whereas the former requires a
coherent lift of the corresponding orbital point.

\subsection{Comparison with orbit-relative LS-category}

The next result relates the two invariants.

\begin{thm}
\label{prop:surjective-LS-category}
\label{prop:category-less-than-sectional}
\label{cor:sectional-equals-category}
Let $\mathcal O$ be an orbit family and let
$p\colon\Ee\to\Bb$ be a morphism of $\CC$-spaces.
\begin{enumerate}
    \item If $p$ is $\mathcal O$-surjective, then
    \[
    \sec_{\CC}(p)\leq\cat_{\CC}^{\mathcal O}(\Bb).
    \]

    \item If $\cat_{\CC}^{\mathcal O}(\Ee)=0$, then
    \[
    \cat_{\CC}^{\mathcal O}(\Bb)\leq\sec_{\CC}(p).
    \]
\end{enumerate}
Consequently, if both hypotheses hold, then
$\sec_{\CC}(p)=\cat_{\CC}^{\mathcal O}(\Bb)$.
\end{thm}

\begin{proof}
Suppose first that $p$ is $\mathcal O$-surjective and let
$\Uu\subseteq\Bb$ be $\mathcal O$-categorical. Choose
$\T\in\mathcal O$ and maps $a\colon\Uu\to\T$ and
$O\colon\T\to\Bb$ such that
$\iota_{\Uu}\simeq_{\CC}O\circ a$. Since $O$ lifts to a map
$\widetilde O\colon\T\to\Ee$, the composite
$s=\widetilde O\circ a$ satisfies
$p\circ s=O\circ a\simeq_{\CC}\iota_{\Uu}$. Thus, every
$\mathcal O$-categorical open subspace is sectional for $p$.

For the second assertion, choose $\T\in\mathcal O$ and maps
$a\colon\Ee\to\T$ and $O\colon\T\to\Ee$ such that
$1_{\Ee}\simeq_{\CC}O\circ a$. If $\Uu\subseteq\Bb$ is sectional,
choose $s\colon\Uu\to\Ee$ with
$p\circ s\simeq_{\CC}\iota_{\Uu}$. Then
\[
\iota_{\Uu}
\simeq_{\CC}
p\circ s
\simeq_{\CC}
p\circ O\circ a\circ s,
\]
so $\Uu$ is $\mathcal O$-categorical in $\Bb$. Hence every sectional
cover is an $\mathcal O$-categorical cover.
\end{proof}

For the canonical families, the first inequality gives the following
immediate consequences.

\begin{cor}\label{cor:sectional-three-families}
Let $p\colon\Ee\to\Bb$ be a morphism of $\CC$-spaces.
\begin{enumerate}
    \item If $p$ is objectwise surjective, then
    $\sec_{\CC}(p)\leq\cat_{\CC}^{\mathrm{rep}}(\Bb)$.

    \item If $p$ is $\mathcal O_{\mathrm{disc}}(\CC)$-surjective, then
    $\sec_{\CC}(p)\leq\cat_{\CC}^{\mathrm{disc}}(\Bb)$.

    \item If $p$ is $\mathcal O_{\mathrm{top}}(\CC)$-surjective, then
    $\sec_{\CC}(p)\leq\cat_{\CC}^{\mathrm{top}}(\Bb)$.
\end{enumerate}
\end{cor}

\subsection{Locality and subadditivity}

For an open $\CC$-subspace $\Vv\subseteq\Bb$, let
\[
p|_{\Vv}\colon p^{-1}(\Vv)\longrightarrow\Vv
\]
denote the restriction of $p$.

\begin{prop}\label{prop:sectional-subadditivity}
Let $\{\Vv_0,\ldots,\Vv_r\}$ be a finite open cover of $\Bb$. Then
\[
\sec_{\CC}(p)
\leq
\sum_{\lambda=0}^{r}
\left(\sec_{\CC}(p|_{\Vv_\lambda})+1\right)-1.
\]
Equivalently,
$\sec_{\CC}(p)\leq
\sum_{\lambda=0}^{r}\sec_{\CC}(p|_{\Vv_\lambda})+r$.
\end{prop}

\begin{proof}
If one of the terms on the right is infinite, there is nothing to prove.
Assume therefore that all of them are finite. For each $\lambda$, choose a sectional cover of $\Vv_\lambda$ for
$p|_{\Vv_\lambda}$ with
$\sec_{\CC}(p|_{\Vv_\lambda})+1$ members. Since
$\Vv_\lambda$ is open in $\Bb$, all members of this cover are open in
$\Bb$, and their homotopy local sections become homotopy local sections
of $p$ after composition with
$p^{-1}(\Vv_\lambda)\hookrightarrow\Ee$. The union of these covers is a
sectional cover of $\Bb$ with
$\sum_{\lambda=0}^{r}(\sec_{\CC}(p|_{\Vv_\lambda})+1)$ members.
\end{proof}

Subdiagrams with empty values provide a useful special case.

\begin{defi}
A full subcategory $\DD\subseteq\CC$ is a \emph{cosieve} if
$c\in\Obj(\DD)$ and $c\to d$ in $\CC$ imply
$d\in\Obj(\DD)$.
\end{defi}

For a cosieve $\DD$, define the open $\CC$-subspace
$\Bb_{\DD}\subseteq\Bb$ by
\[
\Bb_{\DD}(c)=
\begin{cases}
\Bb(c),&c\in\Obj(\DD),\\
\varnothing,&c\notin\Obj(\DD).
\end{cases}
\]

\begin{cor}\label{cor:sectional-cosieve-cover}
If $\CC=\DD_0\cup\cdots\cup\DD_r$ is a finite cover by cosieves, then
\[
\sec_{\CC}(p)
\leq
\sum_{\lambda=0}^{r}
\left(
\sec_{\CC}(p|_{\Bb_{\DD_\lambda}})+1
\right)-1.
\]
\end{cor}

\subsection{Evaluation, limits and global sections}

Evaluation at an object always produces an ordinary lower bound.

\begin{prop}
\label{prop:sectional_category_inequality_objectwise}
\label{prop:sectional-zero-functors}
Let $p\colon\Ee\to\Bb$ be a morphism of $\CC$-spaces.
\begin{enumerate}
    \item For every $c\in\Obj(\CC)$,
    \[
    \sec(p_c)\leq\sec_{\CC}(p).
    \]

    \item If $\sec_{\CC}(p)=0$, then
    $\sec(\varinjlim_{\CC}p)=0$ and, for every $\CC$-space $\Aa$,
    \[
    \sec\bigl(\Hom_{\CC}(\Aa,p)\bigr)=0.
    \]
\end{enumerate}
\end{prop}

\begin{proof}
A sectional cover of $\Bb$ evaluates at $c$ to an open cover of
$\Bb(c)$, and every homotopy local section evaluates to one for $p_c$.

For the second assertion, let $s\colon\Bb\to\Ee$ satisfy
$p\circ s\simeq_{\CC}1_{\Bb}$. Applying either
$\varinjlim_{\CC}$ or $\Hom_{\CC}(\Aa,-)$ produces a homotopy section
of the corresponding map, since both functors preserve homotopies.
\end{proof}

To obtain a bound from an inverse limit, we use the following elementary
condition.

\begin{defi}
An object $i\in\Obj(\CC)$ is \emph{weakly initial} if
$\CC(i,c)\neq\varnothing$ for every $c\in\Obj(\CC)$.
\end{defi}

\begin{prop}\label{prop:sectional_category_limits}
Suppose that $\CC$ has finitely many objects and a weakly initial object.
Then
\[
\sec\left(\varprojlim_{\CC}p\right)
\leq
\sec_{\CC}(p).
\]
\end{prop}

\begin{proof}
Let $\{\Uu_0,\ldots,\Uu_n\}$ be a sectional cover of $\Bb$, and let
$i$ be weakly initial. We claim that the spaces
$\varprojlim_{\CC}\Uu_j$ form an open cover of
$\varprojlim_{\CC}\Bb$.

A point $y=(y_c)_c$ of $\varprojlim_{\CC}\Bb$ satisfies
$y_i\in\Uu_j(i)$ for some $j$. For every object $c$, choose a morphism
$\alpha_c\colon i\to c$. Compatibility gives
$y_c=\Bb(\alpha_c)(y_i)\in\Uu_j(c)$, and hence
$y\in\varprojlim_{\CC}\Uu_j$.

Moreover,
\[
\varprojlim_{\CC}\Uu_j
=
\left(\prod_{c\in\Obj(\CC)}\Uu_j(c)\right)
\cap
\varprojlim_{\CC}\Bb.
\]
The product on the right is open because $\CC$ has finitely many
objects. Finally, a homotopy local section of $p$ over $\Uu_j$ induces,
after taking limits, a homotopy local section of
$\varprojlim_{\CC}p$ over $\varprojlim_{\CC}\Uu_j$.
\end{proof}

Under the hypotheses of the preceding proposition, we therefore have
\[
\max\left\{
\sup_{c\in\Obj(\CC)}\sec(p_c),
\,
\sec\left(\varprojlim_{\CC}p\right)
\right\}
\leq
\sec_{\CC}(p).
\]

\subsection{Homotopy lifting and strict sections}

\begin{defi}\label{def:HLP-diagrams}
A morphism $p\colon\Ee\to\Bb$ has the \emph{homotopy lifting property}
if every commutative square
\[
\begin{tikzcd}
\X
  \arrow[r,"f"]
  \arrow[d,"i_0"']
&
\Ee
  \arrow[d,"p"]
\\
\X\times\mathbb I
  \arrow[r,"H"']
&
\Bb
\end{tikzcd}
\]
admits a lift $\widetilde H\colon\X\times\mathbb I\to\Ee$ satisfying
$p\circ\widetilde H=H$ and $\widetilde H\circ i_0=f$.
\end{defi}

\begin{prop}\label{prop:homotopy-sections-strict}
If $p$ has the homotopy lifting property, every homotopy local section
of $p$ can be replaced by a strict local section. Consequently,
$\sec_{\CC}(p)$ may be defined using strict local sections.
\end{prop}

\begin{proof}
Let $s\colon\Uu\to\Ee$ satisfy
$p\circ s\simeq_{\CC}\iota_{\Uu}$, and choose a homotopy
$H\colon\Uu\times\mathbb I\to\Bb$ from $p\circ s$ to
$\iota_{\Uu}$. Lift $H$ with initial value $s$. Then
$\widetilde H_1\colon\Uu\to\Ee$ satisfies
$p\circ\widetilde H_1=\iota_{\Uu}$.
\end{proof}

\subsection{Domination of the base}

\begin{prop}
\label{prop:homotopic_invariance_base}
\label{cor:sectional-homotopy-base}
Let $f\colon\X\to\Y$ and $g\colon\Y\to\X$ be $\CC$-maps satisfying
$f\circ g\simeq_{\CC}1_{\Y}$. For every morphism
$p\colon\Ee\to\X$,
\[
\sec_{\CC}(f\circ p)\leq\sec_{\CC}(p).
\]
\end{prop}

\begin{proof}
Let $\{\Uu_0,\ldots,\Uu_n\}$ be a sectional cover of $\X$ for $p$.
The subspaces $\V_i=g^{-1}(\Uu_i)$ form an open cover of $\Y$. Let
$s_i\colon\Uu_i\to\Ee$ be homotopy local sections and let
$g_i\colon\V_i\to\Uu_i$ be the restrictions of $g$. If
$\iota_i'$ denotes the inclusion of $\V_i$ into $\Y$, then
\[
f\circ p\circ s_i\circ g_i
\simeq_{\CC}
f\circ g\circ\iota_i'
\simeq_{\CC}
\iota_i'.
\]
Thus, each $\V_i$ is sectional for $f\circ p$.
\end{proof}

\section{Topological complexity of a diagram}
\label{sec:topological-complexity}

The topological complexity of a space is the sectional category of its
endpoint evaluation map. Since path spaces and endpoint evaluations are
defined objectwise, the same construction applies to $\CC$-spaces.

\begin{defi}
\label{def:path-diagram}
\label{def:TC-diagram}
Let $\X$ be a $\CC$-space. Its \emph{path $\CC$-space} is the
$\CC$-space $\X^{\mathbb I}$ defined by
\[
\X^{\mathbb I}(c)=\Hom_{\Top}(\mathbb I,\X(c)),
\qquad
\X^{\mathbb I}(\alpha)(\gamma)=\X(\alpha)\circ\gamma.
\]
The endpoint evaluation is the $\CC$-map
$\pi_{\X}\colon\X^{\mathbb I}\to\X\times\X$ given objectwise by
$\pi_{\X,c}(\gamma)=(\gamma(0),\gamma(1))$.

The \emph{topological complexity} of $\X$ is
\[
\TC_{\CC}(\X)=\sec_{\CC}(\pi_{\X}).
\]
\end{defi}

Thus, $\TC_{\CC}(\X)$ measures the minimum number of open
$\CC$-subspaces required to define compatible local motion planners.
We use the reduced convention, so $\TC_{\CC}(\X)=0$ precisely when
$\pi_{\X}$ admits a global homotopy section. For the terminal indexing
category, this reduces to the ordinary topological complexity.

\subsection{Orbit-relative path-connectivity}

The surjectivity of endpoint evaluation is governed by connectivity of
the spaces of generalized points.

\begin{defi}\label{def:O-path-connected}
Let $\mathcal O$ be an orbit family. A $\CC$-space $\X$ is
\emph{$\mathcal O$-path-connected} if any two points of
$\Hom_{\CC}(\T,\X)$ can be joined by a path, for every
$\T\in\mathcal O$.
\end{defi}

Here path-connectivity is understood in the sense that every pair of
points can be joined by a path; this condition is therefore vacuous for
the empty space. By Proposition~\ref{prop:adjoint}, a path between two
maps $\T\to\X$ is equivalently a $\CC$-homotopy between them.

\begin{prop}
\label{prop:connectivity-monotone}
\label{prop:representable-connected}
\label{prop:pathfibration_surjective}
\label{cor:canonical-connectivity}
Let $\mathcal O$ be an orbit family.
\begin{enumerate}
    \item A $\CC$-space $\X$ is $\mathcal O$-path-connected if and only
    if $\pi_{\X}$ is $\mathcal O$-surjective.

    \item The space $\X$ is
    $\mathcal O_{\mathrm{rep}}(\CC)$-path-connected if and only if
    every $\X(c)$ is path-connected.

    \item If $\mathcal O\subseteq\mathcal O'$ and $\X$ is
    $\mathcal O'$-path-connected, then it is
    $\mathcal O$-path-connected.
\end{enumerate}
Consequently,
\[
\mathcal O_{\mathrm{top}}\text{-path-connected}
\Longrightarrow
\mathcal O_{\mathrm{disc}}\text{-path-connected}
\Longrightarrow
\mathcal O_{\mathrm{rep}}\text{-path-connected}.
\]
\end{prop}

\begin{proof}
For $\T\in\mathcal O$, Propositions~\ref{prop:adjoint} and
the universal property of the product give natural homeomorphisms
\[
\Hom_{\CC}(\T,\X^{\mathbb I})
\cong
\Hom_{\Top}\bigl(\mathbb I,\Hom_{\CC}(\T,\X)\bigr)
\]
and
\[
\Hom_{\CC}(\T,\X\times\X)
\cong
\Hom_{\CC}(\T,\X)^2.
\]
Under these identifications, $\Hom_{\CC}(\T,\pi_{\X})$ is the endpoint
evaluation on the path space of $\Hom_{\CC}(\T,\X)$. It is surjective
precisely when every pair of points can be joined by a path, proving the
first assertion.

The second follows from the enriched Yoneda homeomorphism
$\allowbreak\Hom_{\CC}\allowbreak(\CC(c,-),\allowbreak \X)\allowbreak \cong \allowbreak \X(c)$. The third is immediate from the
inclusion of orbit families.
\end{proof}

For the discrete and topological families, Proposition
\ref{prop:three-surjectivities} shows that path-connectivity can be
tested on the canonical orbital points of $\X\times\X$. This condition
cannot generally be expressed solely in terms of
$\varinjlim_{\CC}\X$, since colimits need not preserve products.

\subsection{Examples of relations between orbits and being orbit-path-connected}
   
An orbit need not be orbit-path-connected.

\begin{ex}
\label{ex:representable-not-connected}
\label{ex:arrow-orbit-not-connected}
Let $\mathcal S$ be the category of two parallel arrows introduced in Example~\ref{ex:category-two-parallel-arrows}. The
representable orbit $\mathcal S(a,-)$ is not
$\mathcal O_{\mathrm{rep}}(\mathcal S)$-path-connected because
\[
\Hom_{\mathcal S}
\bigl(\mathcal S(b,-),\mathcal S(a,-)\bigr)
\cong
\mathcal S(a,b)=\{f_1,f_2\}
\]
is discrete and disconnected.

Similarly, for the category $\mathcal I=(0\to1)$ introduced in Example~\ref{ex:interval}, every diagram
$A\to *$ is a topological orbit. If $A$ is not path-connected, then
this orbit is not
$\mathcal O_{\mathrm{rep}}(\mathcal I)$-path-connected.
\end{ex}

Conversely, orbit-path-connectivity does not imply that the diagram is an
orbit.

\begin{ex}\label{ex:connected-non-orbit}
Consider the $\mathcal I$-space
\[
\X=
\left(
\mathbb R\times\mathbb S^1
\xrightarrow{\pi_2}
\mathbb S^1
\right).
\]
Every topological $\mathcal I$-orbit has the form
$\T=(A\to *)$, and a natural transformation $\T\to\X$ is determined
by a map $A\to\mathbb R$ and a point of $\mathbb S^1$. Hence
\[
\Hom_{\mathcal I}(\T,\X)
\cong
\Hom_{\Top}(A,\mathbb R)\times\mathbb S^1.
\]
The first factor is contractible by linear contraction and the second is
path-connected. Thus, $\X$ is
$\mathcal O_{\mathrm{top}}(\mathcal I)$-path-connected. Nevertheless,
$\varinjlim_{\mathcal I}\X\cong\mathbb S^1$, so $\X$ is not an orbit.
\end{ex}

\subsection{Recovery of the equivariant invariants}
\label{subsec:equivariant-case}

Let $G$ be a discrete group, regarded as a one-object category. Then
$\Top^G$ is the category of topological $G$-spaces and equivariant maps.
The unique representable orbit is the regular $G$-set $G$, while the
discrete orbits are precisely the transitive $G$-sets $G/H$.

Consequently, $\cat_G^{\mathrm{disc}}(X)$ recovers the equivariant
Lusternik--Schnirelmann category studied in
\cite{Marzantowwicz,Colman,Grant-Colman-Equivariant}: an invariant open
subset is categorical when its inclusion can be equivariantly compressed
to an orbit $G/H$.

Similarly, $\sec_G(p)$ is the equivariant sectional category of a
$G$-map $p$, and
\[
\TC_G(X)=\sec_G(\pi_X)
\]
is the equivariant topological complexity of Colman and Grant
\cite{Grant-Colman-Equivariant}, where $G$ acts pointwise on
$X^{\mathbb I}$ and diagonally on $X\times X$.

For every subgroup $H\leq G$, the natural homeomorphism
$\Hom_G(G/H,X) \allowbreak\cong \allowbreak X^H$ shows that
$\mathcal O_{\mathrm{disc}}(G)$-path-connectivity is equivalent to the
path-connectivity of every fixed-point space $X^H$. More generally, a
family $\mathcal F$ of subgroups determines the orbit family
$\{G/H\mid H\in\mathcal F\}$ and the corresponding connectivity
condition on the spaces $X^H$.

\begin{rem}
This comparison is stated for discrete groups. For a non-discrete
topological group, the one-object ordinary category forgets the topology
of the group and does not encode joint continuity of the action. An
enriched version of the present framework would be required in that
setting.
\end{rem}

\subsection{Endpoint evaluation and LS-category}

The endpoint evaluation has the homotopy lifting property.

\begin{prop}\label{prop:path-fibration-HLP}
For every $\CC$-space $\X$, the endpoint evaluation map
\[
\pi_{\X}\colon\X^{\mathbb I}\longrightarrow\X\times\X
\]
has the homotopy lifting property.
\end{prop}

\begin{proof}
Consider a commutative square
\[
\begin{tikzcd}
\Y
  \arrow[r,"f"]
  \arrow[d,"i_0"']
&
\X^{\mathbb I}
  \arrow[d,"\pi_{\X}"]
\\
\Y\times\mathbb I
  \arrow[r,"H"']
&
\X\times\X.
\end{tikzcd}
\]
Write $H=(H^0,H^1)$, where
$H^\varepsilon\colon\Y\times\mathbb I\to\X$ for
$\varepsilon\in\{0,1\}$. The commutativity of the square gives, for
every $c\in\Obj(\CC)$,
\[
H^0_c(y,0)=f_c(y)(0),
\qquad
H^1_c(y,0)=f_c(y)(1).
\]

For every object $c$, define
$\widetilde H_c\colon\Y(c)\times\mathbb I\to\X(c)^{\mathbb I}$ by
\[
\widetilde H_c(y,t)(s)=
\begin{cases}
H^0_c(y,t-3s),
    & 0\leq s\leq \dfrac{t}{3},\\[2mm]
f_c(y)\left(
\dfrac{s-\frac{t}{3}}{1-\frac{2t}{3}}
\right),
    & \dfrac{t}{3}\leq s\leq1-\dfrac{t}{3},\\[4mm]
H^1_c(y,3s+t-3),
    & 1-\dfrac{t}{3}\leq s\leq1.
\end{cases}
\]

The first piece traverses the first endpoint homotopy backwards from
$H^0_c(y,t)$ to $f_c(y)(0)$; the middle piece traverses the original
path $f_c(y)$; and the last piece traverses the second endpoint
homotopy from $f_c(y)(1)$ to $H^1_c(y,t)$. The three expressions agree
at $s=t/3$ and $s=1-t/3$, so the pasting lemma and the exponential law
give the continuity of $\widetilde H_c$.

At $t=0$, the first and third pieces are constant and the middle piece
is $f_c(y)$, hence
$\widetilde H_c(y,0)=f_c(y)$. Moreover,
\[
\widetilde H_c(y,t)(0)=H^0_c(y,t),
\qquad
\widetilde H_c(y,t)(1)=H^1_c(y,t),
\]
and therefore $\pi_{\X,c}\circ\widetilde H_c=H_c$.

Finally, let $\alpha\colon c\to d$ be a morphism of $\CC$. The
naturality of $f$, $H^0$ and $H^1$ implies, on each of the three pieces,
that
\[
\X^{\mathbb I}(\alpha)\circ\widetilde H_c
=
\widetilde H_d\circ
\bigl(\Y(\alpha)\times1_{\mathbb I}\bigr).
\]
Thus, the maps $\widetilde H_c$ assemble into a $\CC$-homotopy
$\widetilde H\colon\Y\times\mathbb I\to\X^{\mathbb I}$ satisfying
$\pi_{\X}\circ\widetilde H=H$ and
$\widetilde H\circ i_0=f$.
\end{proof}

\begin{cor}\label{cor:TC-strict-sections}
In the definition of $\TC_{\CC}(\X)$, homotopy local sections of
$\pi_{\X}$ may be replaced by strict local sections.
\end{cor}

\begin{proof}
This follows from Propositions~\ref{prop:path-fibration-HLP} and
\ref{prop:homotopy-sections-strict}.
\end{proof}
We can now obtain the basic upper bound.

\begin{thm}\label{thm:TC-category-bound}
Let $\mathcal O$ be an orbit family and let $\X$ be
$\mathcal O$-path-connected. Then
\[
\TC_{\CC}(\X)
\leq
\cat_{\CC}^{\mathcal O}(\X\times\X).
\]
\end{thm}

\begin{proof}
By Proposition~\ref{prop:pathfibration_surjective}, $\pi_{\X}$ is
$\mathcal O$-surjective. The result follows from
Theorem~\ref{prop:surjective-LS-category}.
\end{proof}

\begin{cor}\label{cor:TC-three-families}
Let $\X$ be a $\CC$-space.
\begin{enumerate}
    \item If every $\X(c)$ is path-connected, then
    $\TC_{\CC}(\X)\leq
    \cat_{\CC}^{\mathrm{rep}}(\X\times\X)$.

    \item If $\X$ is
    $\mathcal O_{\mathrm{disc}}(\CC)$-path-connected, then
    $\TC_{\CC}(\X)\leq
    \cat_{\CC}^{\mathrm{disc}}(\X\times\X)$.

    \item If $\X$ is
    $\mathcal O_{\mathrm{top}}(\CC)$-path-connected, then
    $\TC_{\CC}(\X)\leq
    \cat_{\CC}^{\mathrm{top}}(\X\times\X)$.
\end{enumerate}
\end{cor}

Thus, enlarging the orbit family gives a potentially smaller
LS-category bound at the cost of a stronger connectivity hypothesis.

\subsection{Lower bounds and homotopy invariance}

\begin{prop}
\label{prop:TC-objectwise}
\label{prop:TC-limit}
Let $\X$ be a $\CC$-space.
\begin{enumerate}
    \item For every $c\in\Obj(\CC)$,
    \[
    \TC(\X(c))\leq\TC_{\CC}(\X).
    \]

    \item If $\CC$ has finitely many objects and a weakly initial
    object, then
    \[
    \TC\left(\varprojlim_{\CC}\X\right)
    \leq
    \TC_{\CC}(\X).
    \]
\end{enumerate}
\end{prop}

\begin{proof}
The first assertion follows by applying
Proposition~\ref{prop:sectional_category_inequality_objectwise} to
$\pi_{\X}$.

For the second, limits commute with products and
Proposition~\ref{prop:adjoint} gives
\[
\varprojlim_{\CC}\X^{\mathbb I}
\cong
\Hom_{\Top}
\left(
\mathbb I,\varprojlim_{\CC}\X
\right)
=
\left(\varprojlim_{\CC}\X\right)^{\mathbb I}.
\]
Under these identifications, $\varprojlim_{\CC}\pi_{\X}$ is the endpoint
evaluation of $\varprojlim_{\CC}\X$. The result follows from
Proposition~\ref{prop:sectional_category_limits}.
\end{proof}

Under the hypotheses of the second assertion, these estimates combine as
\[
\max\left\{
\sup_{c\in\Obj(\CC)}\TC(\X(c)),
\,
\TC\left(\varprojlim_{\CC}\X\right)
\right\}
\leq
\TC_{\CC}(\X).
\]

\begin{prop}\label{prop:TC-domination}
Let $f\colon\X\to\Y$ and $g\colon\Y\to\X$ be $\CC$-maps satisfying
$f\circ g\simeq_{\CC}1_{\Y}$. Then
\[
\TC_{\CC}(\Y)\leq\TC_{\CC}(\X).
\]
In particular, $\CC$-homotopy equivalent diagrams have the same
topological complexity.
\end{prop}

\begin{proof}
Since $(f\times f)\circ(g\times g)\simeq_{\CC}1_{\Y\times\Y}$,
Proposition~\ref{prop:homotopic_invariance_base} gives
\[
\sec_{\CC}\bigl((f\times f)\circ\pi_{\X}\bigr)
\leq
\sec_{\CC}(\pi_{\X}).
\]
Naturality of endpoint evaluation gives
$\pi_{\Y}\circ f^{\mathbb I}=(f\times f)\circ\pi_{\X}$. Hence every
homotopy local section of $(f\times f)\circ\pi_{\X}$ becomes one of
$\pi_{\Y}$ after composition with $f^{\mathbb I}$. Therefore,
\[
\TC_{\CC}(\Y)
\leq
\sec_{\CC}\bigl((f\times f)\circ\pi_{\X}\bigr)
\leq
\TC_{\CC}(\X).
\]
Applying the inequality in both directions gives homotopy invariance.
\end{proof}

\subsection{Examples}

\begin{ex}\label{ex:contractible-values-nontrivial-TC}
Let $\mathcal S$ be the category of two parallel arrows and consider
\[
\begin{tikzcd}
\mathbb R^2
  \arrow[r,"q",bend left]
  \arrow[r,"1"',bend right]
&
\mathbb R,
\end{tikzcd}
\qquad
q(x,y)=x^2+y^2,
\]
where $1$ is the constant map with value $1$. Its inverse limit is the
equalizer of $q$ and $1$, hence is homeomorphic to $\mathbb S^1$.
Although both values of the diagram are contractible,
\[
\TC_{\mathcal S}(\X)
\geq
\TC(\mathbb S^1)=1.
\]
Thus, the topological complexity of a diagram is not determined by the
topological complexities of its individual values.
\end{ex}

\begin{ex}\label{ex:inverse-limit-without-weakly-initial}
The weakly initial object hypothesis in
Proposition~\ref{prop:TC-limit} cannot be omitted. Let
$\mathcal V=(0\to2\leftarrow1)$ and consider the $\mathcal V$-space
\[
\X=
\left(
\mathbb S^1\longrightarrow *\longleftarrow\mathbb S^1
\right).
\]
The local motion planners on the two copies of $\mathbb S^1$ may be
chosen independently and grouped into a common two-member cover, since
compatibility at the terminal value is automatic. Hence
\[
\TC_{\mathcal V}(\X)=1.
\]
On the other hand,
\[
\varprojlim_{\mathcal V}\X
\cong
\mathbb S^1\times\mathbb S^1,
\]
and therefore
\[
\TC\left(\varprojlim_{\mathcal V}\X\right)=2.
\]
Thus,
\[
\TC\left(\varprojlim_{\mathcal V}\X\right)
>
\TC_{\mathcal V}(\X).
\]
\end{ex}

\begin{ex}\label{ex:colimit-does-not-give-lower-bound}
There is no analogous lower bound obtained from the colimit. Let
$\Lambda=(0\leftarrow2\to1)$ and consider
\[
\Y=
\left(
\mathbb S^1\longleftarrow *\longrightarrow\mathbb S^1
\right),
\]
where both maps select the basepoint. Compatible local motion planners
on the two circles give
\[
\TC_{\Lambda}(\Y)=1.
\]
However,
\[
\varinjlim_{\Lambda}\Y
\cong
\mathbb S^1\vee\mathbb S^1,
\]
whose reduced topological complexity is $2$. Consequently,
\[
\TC_{\Lambda}(\Y)
<
\TC\left(\varinjlim_{\Lambda}\Y\right).
\]

The obstruction is structural: colimits need not preserve products.
The diagrammatic endpoint evaluation only involves pairs of states
belonging to a common value of the diagram, whereas the square of the
colimit also contains pairs represented at different objects.
\end{ex}
\section{Change of base}
\label{sec:change-of-base}

Let $F\colon\CC\to\Dc$ be a functor. Precomposition defines the
\emph{change-of-base functor}
\[
F^*\colon\Top^{\Dc}\longrightarrow\Top^{\CC},
\qquad
F^*\X=\X\circ F.
\]
Thus, $F^*\X(c)=\X(F(c))$ and
$F^*\X(\alpha)=\X(F(\alpha))$.

Since limits and colimits in functor categories are computed objectwise,
$F^*$ preserves all objectwise constructions. In particular, it preserves
products, constant diagrams and homotopies. It also preserves open
subspaces and open covers: if $\{\Uu_i\}_{i\in I}$ is an open cover of a
$\Dc$-space $\X$, then $\{F^*\Uu_i\}_{i\in I}$ is an open cover of
$F^*\X$.

Indeed, for every $c\in\Obj(\CC)$,
\[
F^*\Uu_i(c)=\Uu_i(F(c))
\quad\text{and}\quad
\bigcup_{i\in I}F^*\Uu_i(c)=\X(F(c))=F^*\X(c).
\]
Similarly, if $H\colon\X\times\mathbb I\to\Y$ is a $\Dc$-homotopy from
$f$ to $g$, then $F^*H$ is a $\CC$-homotopy from $F^*f$ to $F^*g$,
using the natural identification
$F^*(\X\times\mathbb I)=F^*\X\times\mathbb I$.

\subsection{Final functors and orbit families}

Restriction along an arbitrary functor need not preserve orbits. The
appropriate condition is finality.

\begin{defi}\label{def:comma-d-F}
For $d\in\Obj(\Dc)$, the comma category $d\downarrow F$ has as objects
the pairs $(c,u)$, where $c\in\Obj(\CC)$ and
$u\colon d\to F(c)$. A morphism
$(c,u)\to(c',u')$ is a morphism $\alpha\colon c\to c'$ in $\CC$
satisfying $F(\alpha)\circ u=u'$.

The functor $F$ is \emph{final} if $d\downarrow F$ is nonempty and
connected for every $d\in\Obj(\Dc)$.
\end{defi}

Equivalently, \(F\) is final if, for every \(\Dc\)-space \(\X\), the
canonical map
\[
\varinjlim_{\CC}F^*\X\longrightarrow\varinjlim_{\Dc}\X
\]
is an isomorphism \cite[Chapter~IX, Section~3]{MacLaneCategories}.

\begin{prop}
\label{prop:final-colimits}
\label{prop:final-preserves-top-orbits}
\label{prop:final-preserves-disc-orbits}
Let $F\colon\CC\to\Dc$ be final.
\begin{enumerate}
    \item If $\T$ is a topological $\Dc$-orbit, then $F^*\T$ is a
    topological $\CC$-orbit.

    \item If $\T$ is a discrete $\Dc$-orbit, then $F^*\T$ is a
    discrete $\CC$-orbit.
\end{enumerate}
\end{prop}

\begin{proof}
Finality gives
$\varinjlim_{\CC}F^*\T\cong\varinjlim_{\Dc}\T$. If $\T$ is a
topological orbit, the latter space is a point, proving the first
assertion.

If $\T$ is discrete, then $F^*\T$ is also objectwise discrete and
finality in $\Set$ gives
$\varinjlim_{\CC}^{\Set}F^*\T
\cong\varinjlim_{\Dc}^{\Set}\T\cong *$.
\end{proof}

\begin{rem}\label{rem:representables-change-base}
Finality does not generally preserve representable orbits. For
$d\in\Obj(\Dc)$,
\[
F^*\Dc(d,-)=\Dc(d,F(-)),
\]
which need not be representable in $\Set^{\CC}$. Additional hypotheses
are therefore required for the representable orbit family.
\end{rem}

Let $\mathcal O$ be an orbit family for $\Dc$ and $\mathcal P$ an orbit
family for $\CC$. We write
$F^*\mathcal O\subseteq\mathcal P$ when
$F^*\T\in\mathcal P$ for every $\T\in\mathcal O$.

\begin{thm}
\label{prop:change-base-cover}
\label{prop:change-base-homotopy}
\label{prop:change-base-categorical-open}
\label{thm:change-base-LS-category}
Suppose that $F^*\mathcal O\subseteq\mathcal P$. Then, for every
$\Dc$-space $\X$,
\[
\cat_{\CC}^{\mathcal P}(F^*\X)
\leq
\cat_{\Dc}^{\mathcal O}(\X).
\]
\end{thm}

\begin{proof}
Let $\Uu$ be $\mathcal O$-categorical in $\X$. Choose
$\T\in\mathcal O$ and maps $a\colon\Uu\to\T$ and
$O\colon\T\to\X$ such that
$\iota_{\Uu}\simeq_{\Dc}O\circ a$. Applying $F^*$ gives
\[
F^*\iota_{\Uu}
\simeq_{\CC}
F^*O\circ F^*a.
\]
The map $F^*\iota_{\Uu}$ is the inclusion of the open $\CC$-subspace
$F^*\Uu$ into $F^*\X$, and $F^*\T\in\mathcal P$. Hence $F^*\Uu$ is
$\mathcal P$-categorical.

Applying this observation to every member of an
$\mathcal O$-categorical cover of $\X$ gives a
$\mathcal P$-categorical cover of $F^*\X$ with the same number of
members.
\end{proof}

For the canonical discrete and topological families, finality provides
the required inclusion of orbit classes.

\begin{cor}\label{cor:change-base-canonical-LS}
Let $F\colon\CC\to\Dc$ be final. Then
\[
\cat_{\CC}^{\mathrm{disc}}(F^*\X)
\leq
\cat_{\Dc}^{\mathrm{disc}}(\X),
\qquad
\cat_{\CC}^{\mathrm{top}}(F^*\X)
\leq
\cat_{\Dc}^{\mathrm{top}}(\X).
\]
\end{cor}

For the representable family, we obtain the following conditional
version.

\begin{cor}\label{cor:change-base-representable-LS}
Suppose that $\Dc(d,F(-))$ is isomorphic to a representable
$\CC$-space for every $d\in\Obj(\Dc)$. Then, for every $\Dc$-space $\X$.
\[
\cat_{\CC}^{\mathrm{rep}}(F^*\X)
\leq
\cat_{\Dc}^{\mathrm{rep}}(\X).
\]
\end{cor}

\subsection{Sectional category and topological complexity}

Unlike orbit-relative LS-category, sectional category behaves well under
arbitrary change of base; no finality assumption is needed.

\begin{prop}
\label{prop:change-base-sectional}
\label{prop:change-base-path-space}
\label{thm:change-base-TC}
Let $F\colon\CC\to\Dc$ be any functor.
\begin{enumerate}
    \item For every morphism $p\colon\Ee\to\Bb$ of $\Dc$-spaces,
    \[
    \sec_{\CC}(F^*p)\leq\sec_{\Dc}(p).
    \]

    \item For every $\Dc$-space $\X$,
    \[
    \TC_{\CC}(F^*\X)\leq\TC_{\Dc}(\X).
    \]
\end{enumerate}
\end{prop}

\begin{proof}
Let $\{\Uu_0,\ldots,\Uu_n\}$ be a sectional cover of $\Bb$ for $p$,
with homotopy local sections $s_i\colon\Uu_i\to\Ee$. The subspaces
$F^*\Uu_i$ form an open cover of $F^*\Bb$, and
\[
F^*p\circ F^*s_i
=
F^*(p\circ s_i)
\simeq_{\CC}
F^*\iota_i.
\]
Thus, every $F^*\Uu_i$ is sectional for $F^*p$, proving the first
assertion.

For the second, there are natural identifications
\[
F^*(\X^{\mathbb I})=(F^*\X)^{\mathbb I},
\qquad
F^*(\X\times\X)=F^*\X\times F^*\X,
\]
under which $F^*\pi_{\X}=\pi_{F^*\X}$. Therefore,
\[
\TC_{\CC}(F^*\X)
=
\sec_{\CC}(F^*\pi_{\X})
\leq
\sec_{\Dc}(\pi_{\X})
=
\TC_{\Dc}(\X).
\]
\end{proof}

\subsection{Restriction to subcategories}

Let $\mathcal I\subseteq\CC$ be a subcategory and let
$j\colon\mathcal I\hookrightarrow\CC$ be the inclusion. Then
$j^*\X=\X|_{\mathcal I}$.

\begin{cor}
\label{cor:restriction-sectional}
\label{cor:subcategory-limit-sectional}
\label{cor:restriction-TC}
\label{cor:subcategory-limit-TC}
Let $\mathcal I\subseteq\CC$ be a subcategory.
\begin{enumerate}
    \item For every morphism $p\colon\Ee\to\Bb$ of $\CC$-spaces,
    \[
    \sec_{\mathcal I}(p|_{\mathcal I})
    \leq
    \sec_{\CC}(p).
    \]

    \item For every $\CC$-space $\X$,
    \[
    \TC_{\mathcal I}(\X|_{\mathcal I})
    \leq
    \TC_{\CC}(\X).
    \]

    \item If $\mathcal I$ has finitely many objects and a weakly initial
    object, then
    \[
    \sec\left(\varprojlim_{\mathcal I}p|_{\mathcal I}\right)
    \leq
    \sec_{\CC}(p)
    \]
    and
    \[
    \TC\left(\varprojlim_{\mathcal I}\X|_{\mathcal I}\right)
    \leq
    \TC_{\CC}(\X).
    \]
\end{enumerate}
\end{cor}

\begin{proof}
The first two assertions follow from the preceding proposition applied
to the inclusion $j$. For the last two, combine these restriction
inequalities with Propositions~\ref{prop:sectional_category_limits} and
\ref{prop:TC-limit}, applied over $\mathcal I$.
\end{proof}

Consequently,
\[
\sup_{\substack{
\mathcal I\subseteq\CC\\
\mathcal I\text{ finite}\\
\mathcal I\text{ has a weakly initial object}
}}
\TC\left(\varprojlim_{\mathcal I}\X|_{\mathcal I}\right)
\leq
\TC_{\CC}(\X).
\]
\bibliographystyle{plain}
\bibliography{biblio}

\end{document}